\documentclass{article}

\usepackage [latin1]{inputenc}

\usepackage{amsmath,amsfonts, amsthm, amssymb, version, enumerate}

\numberwithin{equation}{section}

\newcommand\R{\mathbb{R}}
\newcommand\A{\mathfrak{A}}
\renewcommand\P{\mathcal{P}}
\newcommand\N{\mathbb{N}}
\newcommand\Z{\mathbb{Z}}
\newcommand\norm[1]{\Vert #1 \Vert}
\newcommand\pnorm[2]{{\Vert #1 \Vert}_{#2}}

\newtheorem{Theorem}{Theorem}[section]
\newtheorem{Corollary}[Theorem]{Corollary}
\newtheorem{Proposition}[Theorem]{Proposition}
\newtheorem{Example}[Theorem]{Example}
\newtheorem{Lemma}[Theorem]{Lemma}

\theoremstyle{remark}
\newtheorem{Remark}[Theorem]{Remark}

\theoremstyle{definition}
\newtheorem{Definition}[Theorem]{Definition}
\newtheorem{Notation}[Theorem]{Notation}

\begin{document}

\title{On the second parameter of an\\ $(m, p)$-isometry}
\author{Philipp Hoffmann\thanks{Philipp Hoffmann is currently a PhD-student at the School of Mathematical Sciences,
University College Dublin (National University of Ireland). He is funded by the school as a Research
Demonstrator}, Michael Mackey and Mícheál Ó Searcóid}

\date{November 3, 2011}
\maketitle

\begin{center}
\emph{Appeared in\footnote{The final publication is available at 
www.springerlink.com \\ http://www.springerlink.com/content/jx47988153826m44/}: Integr. Equ. Oper. Theory 71 (2011), 389-405\\$\textrm{}$\\$\textrm{}$\\}
\end{center}

\begin{abstract}
  A bounded linear operator $T$ on a Banach space $X$ is called an
  $(m, p)$-isometry if it satisfies the equation
  $\sum_{k=0}^{m}(-1)^{k} {m \choose k}\|T^{k}x\|^{p} = 0$, for all $x
  \in X$. In this paper we study the structure which
  underlies the second parameter of $(m, p)$-isometric operators.  We
  concentrate on determining when an $(m, p)$-isometry is a $(\mu,
  q)$-isometry for some pair ($\mu, q)$.   We also extend
  the definition of $(m, p)$-isometry, to include $p=\infty$ and study
  basic properties of these $(m, \infty)$-isometries.
\end{abstract}

{\bf Keywords:}\ Banach space, operator, $m$-isometry, $(m, p)$-isometry

\pagestyle{myheadings}
\markright{On the second parameter of an $(m, p)$-isometry}

\section{Introduction}
Let $H$ be a Hilbert space over $\mathbb{K}\in\{\mathbb{R,C}\}$, $B(H)$ the algebra
of bounded linear operators $T$ on $H$. Let further
the symbol $\mathbb{N}$ denote the natural numbers including $0$ and
denote $\mathbb{N}^+:=\mathbb{N}\setminus \{0\}$.  For $m \in
\mathbb{N}^+$ an operator $T \in B(H)$ is called an \emph{$m$-isometry} if
\begin{align}\label{original}
	\sum_{k=0}^{m}(-1)^{k}{m \choose k} T^{*k}T^{k}=0,
\end{align}
where $T^*$ denotes the Hilbert adjoint of $T$. (We exclude the
trivial case of $m=0$.)
This condition is obviously a generalization of the notion of an
isometry. In particular, $1$-isometries are isometries and an isometry
is an $m$-isometry for every $m \in \mathbb{N}^+$. Operators of this
type have been introduced (for $m=2$) by Agler in \cite{ag} and,
independently, by Richter in \cite{ri}. They have been studied
extensively for general $m$ by Agler and Stankus in three papers
\cite{ag-st1, ag-st2, ag-st3}. Basic properties of $m$-isometries
include the facts that every $m$-isometry is an $(m+1)$-isometry,
that $m$-isometries are bounded below and that their spectrum
$\sigma(T) \subseteq \mathbb{K}$ lies in the closed unit disc. The
dynamics of $m$-isometric operators have been studied in
\cite{bemarmart} and 
\cite{Ah-He}.

If $H$ is a complex Hilbert space, condition~\eqref{original} can be
rewritten as
\begin{align}\label{m-isometry SA}
	\sum_{k=0}^{m}(-1)^{k}{m \choose k} \|T^{k}x\|^{2}&=0, \ \forall x \in H,
\end{align}
and this formulation can be interpreted in an arbitrary Banach space.
Operators which satisfy (\ref{m-isometry SA}) have been studied by Sid
Ahmed on complex Banach spaces in \cite{sa}
and on general function spaces and $\ell_p$ spaces by Botelho in \cite{bo}.
In a Banach space, however, there is no intrinsic motivation for the
square of the norm which appears in the definition of an
$m$-isometry. Bayart \cite{ba} introduced  a more general term: an
operator $T \in B(X)$ (where $X$ is a Banach space over $\mathbb{K}$)
is called an \emph{$(m, p)$-isometry}, if there exist an $m \in \mathbb{N}^+$
and a $p \in [1, \infty)$, such that
\begin{align}\label{Bayart}
	\sum_{k=0}^{m}(-1)^{k}{m \choose k} \|T^{k}x\|^{p}=0, \ \ \forall x\in X.
\end{align}
Bayart showed that all basic properties of $m$-isometries on
Hilbert spaces (which we should now refer to as $(m, 2)$-isometries)
carry over to $(m, p)$-isometries on Banach spaces and, further, that
$(m, p)$-isometries  are never
$N$-supercyclic if $X$ is of infinite dimension and complex.
In this paper, we do not impose the restriction $p\ge 1$ and thus will call an operator $T \in B(X)$ an
\emph{$(m,p)$-isometry}, if there exists an $m \in \N^+$ and a $p \in (0,\infty)$, such that (\ref{Bayart}) is
satisfied. Most results in the literature remain valid, with their existing proofs, for $p$ in this
extended range. (Example
\ref{example(2,p)} is an exception.)  Nevertheless, we have found it
appropriate to include short alternative proofs of several known
results.  

Most examples of non-trivial $(m, p)$-isometries in the literature are given in the setting of Hilbert space and 
satisfy (\ref{m-isometry SA}). Botelho gives some examples involving $\ell_2$ direct sums of Banach
spaces in \cite{bo}.  We will refer to two motivating examples from \cite{sa}.

\begin{Example}\label{example finite} 
$T = \left(\begin{array}{cc}
      1 & 1\\
      0 & 1\\
    \end{array} \right)$ on $(\mathbb{C}^{2}, \pnorm{.}{2})$ is a $(3, 2)$-isometry.
\end{Example}  

It is easy to check that this example is neither a $(3, 1)$-isometry,
nor a $(3, 3)$-isometry.  One may ask whether this
operator is a $(3, p)$-isometry for any $p \neq 2$. We will see that it is not, but that $T$ is a $(5, 4)$-isometry, a
$(7,6)$-isometry, a $(9,8)$-isometry and so on.

In the first part of this paper, we will concentrate on determining
the pairs 
 $(\mu, q)$, with $\mu \in \mathbb{N}$, $\mu \geq 2$
and $q \in (0,\infty)$, for which an $(m, p)$-isometry is also a $(\mu,
q)$-isometry.\\

The second example (\cite[Example 2.4]{sa}) is a weighted right shift
operator on the Hilbert space $\ell_2$. The example can be adapted to
 general $\ell_p$ spaces, $1 \leq p < \infty$.

\begin{Example}\label{example(2,p)}
Let $p \in [1, \infty)$, $\lambda \geq 1$ and $T_{p} \in B(\ell_{p})$ be a weighted right-shift operator with 
weight sequence $(\lambda_{n})_{n\in \mathbb{N}}$. That is, for $x = (x_k)_{k \in \mathbb{N}} \in \ell_p$
\begin{align*}
	(T_{p}x)_{n} = \left\{ \begin{array} {ll}
     0, \ \ \ &\textrm{if} \ n = 0,
     \\[1ex]
     \lambda_{n}x_{n-1}, \ \ \ &\textrm{if}\ n \geq 1
\end{array} \right.
\end{align*}
where $(\lambda_{n})_{n\in \mathbb{N}}$ is given by
\begin{align*}
	\lambda_{n}=\bigg{(}\frac{1+(n+1)(\lambda^{2}-1)}{1+n(\lambda^{2}-1)}\bigg{)}^{1/p}.
\end{align*}
(Note that for $\lambda \geq 1$, this is well-defined.) Then $T_{p}$
is a $(2, p)$-isometry. 
This follows from the fact that $(\lambda_{n+2}\lambda_{n+1})^{p}-2(\lambda_{n+1})^{p} = -1$.\\
Remark that for $\lambda >1$ we have $\|T_px\|>\|x\|$, for all
non-zero $x$ and, in particular, $T_p$ is not 
an isometry for $\lambda >1$.
\end{Example}

One may now like to consider the analogue of this example in $\ell_{\infty}$. As
$p$ tends to infinity, the weight sequence becomes a sequence of ones
and $T_{\infty}$ will just be the right shift operator, which is an
isometry. However, it is natural to seek a non-isometric operator
which is a $(2, \infty)$-isometry - a term which so far has not been defined.
In the later parts of this paper, we will give a natural definition of
$(m, \infty)$-isometries and we will show that their basic properties
coincide with those of $(m, p)$-isometries.

\section{Preliminaries}

In this section, we cite some basic results concerning $(m,p)$-isometric
operators proven
by Bayart in \cite{ba}.   There, it is assumed that $p \geq 1$, but on
inspection it is clear that that this restriction is unnecessary and
one can allow $p \in (0,\infty)$. However, we will provide alternative
proofs for (\ref{poly}) and Propositions \ref{power bounded} and \ref{inv} in the next
section. (Then (\ref{beta_m-1}) follows and (\ref{page 3 ba}) obviously does not depend on the range of $p$.)

For $T \in B(X)$, $p \in (0, \infty)$ and $l \in \mathbb{N}$ define the functions
$\beta_{l}^{(p)}(T, \cdot): X \rightarrow \mathbb{R}$, by
\begin{align*}
	\beta_{l}^{(p)}(T, x):= \frac{1}{l!}\sum_{j=0}^{l}(-1)^{l -j}{l \choose j}\|T^{j}x\|^{p}.
\end{align*}
(This is analogous to \cite{ba}.) $T$ is an $(l, p)$-isometry, iff 
$\beta^{(p)}_{l}(T, \cdot)\equiv 0$.

For $k, n \in \mathbb{N}$ denote the \emph{(descending) Pochhammer symbol} by $n^{(k)}$, i.e.
\begin{align*}
	n^{(k)}:= \left\{ \begin{array} {lll}
					1, \ \ &\textrm{if} \ &n = 0,\\[1ex]
					0, &\textrm{if} \ &n > 0, \ k > n,\\[1ex]
       				{n \choose k}k!, &\textrm{if} \ &n > 0, \ k \leq n.
       				\end{array} \right.
\end{align*}
Then for $n > 0$, $k > 0$ and $k \leq n$ we have
\begin{align*}
	n^{(k)} = n(n-1)\cdots(n-k+1).
\end{align*}
By \cite[(2)]{ba} we have, for all $T \in B(X)$ and for all $n \in \mathbb{N}$,
\begin{align}\label{general poly}
	\|T^{n}x\|^{p}=\sum_{k=0}^{n}n^{(k)}\beta^{(p)}_{k}(T, x), \ \ \forall x \in X.
\end{align}
Further, by \cite[page 3]{ba}, the functions $\beta_{l}^{(p)}(T, \cdot)$ satisfy
\begin{align}\label{page 3 ba}
	l!\big(\beta_{l}^{(p)}(T, Tx)-\beta_{l}^{(p)}(T, x)\big) = (l+1)!\beta_{l+1}^{(p)}(T, x), 
		\ \ \forall x \in X.
\end{align}

This shows that an $(m, p)$-isometry is an $(m+1, p)$-isometry. Hence, if $T$ is an $(m ,p)$-isometry, we have \cite[Proposition 2.1]{ba}, on discarding terms from (\ref{general poly}), that
\begin{align}\label{poly}
	\|T^{n}x\|^{p}=\sum_{k=0}^{m-1}n^{(k)}\beta^{(p)}_{k}(T, x), \ \ \forall x \in X \ \textrm{and all} \ 
	n \in \mathbb{N}.
\end{align}
Finally, we get equation \cite[(4)]{ba} for an $(m, p)$-isometry $T$:
\begin{align}\label{beta_m-1}
	\lim_{n\rightarrow \infty} \frac{\|T^{n}x\|^{p}}{n^{m-1}} = \beta_{m-1}^{(p)}(T, x), \ \ \forall x \in X.
\end{align}
This implies the following useful proposition (see also \cite[Corollary 2.4]{bemarmart}):
\begin{Proposition}\label{power bounded}
Let $T \in B(X)$ be an $(m, p)$-isometry such that for each $x \in X$ there exists a real number 
$C(x) > 0$, with
\begin{align*}
	\|T^{n}x\| \leq C(x), \ \forall n \in \mathbb{N}.
\end{align*}
Then $T$ is an isometry.
\end{Proposition}
\begin{proof}
Since $T$ is an $(m, p)$-isometry, we have
\begin{align*}
	0 \leq \beta_{m-1}^{(p)}(T, x)=
	\lim_{n \rightarrow \infty}\frac{\|T^{n}x\|^{p}}{n^{m-1}} \leq 
	\lim_{n \rightarrow \infty}\frac{(C(x))^{p}}{n^{m-1}} 
	= 0, \ \forall x \in X.
\end{align*}
Thus, $T$ is an $(m-1, p)$-isometry.
Applying the same argument sufficiently often gives that $T$ is an isometry.
\end{proof}

Observe that we have not made use of the uniform boundedness principle above. We deliberately avoid appealing 
to the linearity or continuity of $T$ where possible.\\ 

Using (\ref{beta_m-1}), Bayart showed the following in \cite[Proposition 2.4.(b)]{ba}:
\begin{Proposition}\label{inv}
If $T \in B(X)$ is an invertible $(m, p)$-isometry and $m$ is even then $T$ is an
$(m-1, p)$-isometry.
\end{Proposition}

This implies in particular, since $(m, p)$-isometries are bounded
below by definition (see also \cite[Lemma 2.1]{sa}), that
there are no non-trivial examples for $(2,p)$-isometries in finite dimensions.\\

Since referring to an operator as an $(m, p)$-isometry does not
exclude the possibility of this operator being an $(m-1, p)$-isometry,
some statements could become convoluted and this motivates
the following terminology.

\begin{Definition}
Let $T \in B(X)$ be an $(m, p)$-isometry and not an $(m-1, p)$-isometry. Then we call $T$ a 
\emph{strict $(m, p)$-isometry} or say that $T$ is $(m, p)$\emph{-strict}.
\end{Definition}

\section{$(m,p)$-isometries - an alternative approach}

There is little about the basic theory of
$(m, p)$-isometries which depends on linearity or continuity of the operator
in question.  Therefore, we find it useful to present an
alternative and perhaps more natural approach which
simplifies many of the proofs. 

\begin{Notation}
Throughout, $\P^n$ denotes the space of real polynomial functions of
degree less than or equal to $n$. Let $\mathfrak F$ denote
the set of real functions whose domain is a subset of $\R$ that is invariant under the
mapping $s:x\rightarrow x+1$. We define $D$ on $\mathfrak F$
by setting $Df:=f-(f\circ s)$ for each $f\in\mathfrak F$.
Note that $Df\in\mathfrak F$, so that $D:\mathfrak F\rightarrow\mathfrak F$ and we can
form successive iterates $D^nf$ on $\mathfrak F$. Then $D^mf = \sum_{k=0}^{m}(-1)^{k}\binom{m}{k}(f\circ s^k)$ 
for all $f\in\mathfrak F$ and $m \in \mathbb{N}$. 

Let now $\A$ be the set of all real sequences and note that $\A\subseteq\mathfrak F$ and $D(\A)\subseteq\A$ and 
$D^ma = \big(\sum_{k=0}^{m}(-1)^{k}\binom{m}{k}a_{n+k}\big)_{n \in \mathbb{N}}$, for all $a\in\A$ and 
$m \in \mathbb{N}$.\\
Denote by $\A^+$ the positive cone of $\mathfrak A$. For $a=(a_n)_{n \in \mathbb{N}}\in
\mathfrak A^+$ and $p\in(0,\infty)$, the sequence $(a_n^p)_{n \in \mathbb{N}}\in \mathfrak A^+$ will be denoted 
by $a^p$. Then, for each $p \in (0, \infty)$ and $m \in \mathbb{N}^+$, we define
\begin{itemize}
	\item[(i)]  
		$\A_{m,p}:=\{a \in \A^+ \mid D^ma^p=0\}$,
	\item[(ii)]
		$\widehat{\A}_{m,p}:=\{a \in \A^+ \mid D^ma^p=0 \ \ \textrm{and} \ \ D^{m-1}a^p \neq 0\}$;
\end{itemize}
and, for each $a\in\A^+$, we define
\begin{itemize}
	\item[(iii)]
		$\rho(a):=\{(m,p)\mid m\in\N^+, p\in(0,\infty), a\in\A_{m,p}\}$,
	\item[(iv)]
		$\hat\rho(a):=\{(m,p)\mid m\in\N^+, p\in(0,\infty), a\in\widehat\A_{m,p}\}$.
\end{itemize}
\end{Notation}

The following result is surely well-known but, lacking a reference,
we provide a proof for completeness.

\begin{Proposition}\label{altlemm}
\begin{enumerate}[(i)]
	\item
		$D(\P^0)=\{0\}$ and $D(\P^{m+1})=\P^{m}$ for all $m\in\mathbb{N}$; 
	\item
		for $a\in\A$ and $m\in\mathbb{N}$, $D^{m+1}a=0$ if, and only if, there exists $P\in\P^{m}$ such that
		$a=P|_{\mathbb{N}}$. Moreover, at most one real polynomial function $P$ satisfies $a=P|_{\mathbb{N}}$. 
\end{enumerate}
\end{Proposition}
\begin{proof}
Suppose $m\in\N$. It is easy to check that $D(\P^0)=\{0\}$ and that  $D(\P^{m+1})\subseteq  \P^{m}$.  For the reverse inclusion, suppose $P\in\P^{m}\backslash\{0\}$. We want to
solve the equation $P=Q-(Q\circ s)$ for some $Q\in \P^{d+1}$, where $d$ is the degree of $P$.
Every real polynomial function has a unique representation as a linear combination of power functions; 
equating the coefficients of those power functions on the two sides of the equation $P=Q-(Q\circ s)$, we get $d+1$ linear equations in the $d+1$ unknown coefficients (excluding the constant term) of the
proposed polynomial function $Q$. Observe that each of the equations has a
different number of the unknowns in it,
 making the equations linearly independent. So there are solutions for $Q$, giving  $\P^{m}\subseteq D(\P^{m+1})$, as
required, and we have proved (i).\\
By iteration, (i) gives $D^{m+1}(\P^m)=\{0\}$ for all $m\in\mathbb N$ and the 
backward implication of (ii) follows. The  forward implication is
certainly true for
$m=0$; we suppose it true for $m=k\in\mathbb N$. Consider $a\in\A$ and suppose $D^{k+2} a=0$.
Then $D^{k+1}(Da)=0$ and, by hypothesis, there exists  $P\in\P^k$ with $Da=P|_{\mathbb N}$. By (i), there
exists $Q\in\P^{k+1}$ with $P=DQ$. Then $Da=(DQ)|_{\mathbb N}=D(Q|_{\mathbb N})$, giving 
$D(a-Q|_{\mathbb N})=0$, and then there is a constant function $C:\R\rightarrow\R$ such that
$a-Q|_{\mathbb N}=C|_{\mathbb N}$. So $a=(Q+C)|_{\mathbb N}$. Since $Q+C\in\P^{k+1}$,
the result follows by induction. Uniqueness of $P$ is ensured by the
fact that a polynomial function is fully determined by its values on
an infinite set.
\end{proof}

\begin{Corollary}\label{polyTheorem}
A sequence $a$ is an element of $\A_{m,p}$ if, and only if, there exists $P\in \P^{m-1}$ such that $P|_\N= a^p$ 
(that is, $P(n)=a_n^p$ for all $n\in \N$); in this case, the polynomial function $P$
is uniquely determined by the equation $P|_\N= a^p$.
\end{Corollary}

\begin{Remark}\label{hatrmk}
It is immediate from this characterisation that $a\in \A_{m,p}$ implies $a\in \A_{m+1,p}$.  Thus $(\A_{n,p})_{n \in \N}$
is an increasing sequence of sets and 
$\A_{m,p}= \widehat{\A}_{m,p}\ \dot{\cup} \ \widehat{\A}_{m-1,p}\ \dot{\cup} \ \cdots
\ \dot{\cup} \ \widehat{\A}_{1,p} \ \dot{\cup} \ \{0\}$.
\end{Remark}

\begin{Corollary}\label{prop:strictdegree}
Suppose $a\in \A_{m,p}$ and let $P\in \P^{m-1}$ be the unique polynomial determined by $P|_\N= a^p$ in
\ref{polyTheorem}. Then $a\in \widehat{\A}_{d+1,p}$ where $d$ is the degree of $P$. Conversely, if 
$a\in \widehat{\A}_{m,p}$, the degree of $P$ is $m-1$.  Moreover, $P$ has positive leading coefficient.
\end{Corollary}
\begin{proof}
The first parts are a clear corollary of our notation. That the leading coefficient of $P$ is positive follows from the
fact that $a^p$ is a positive sequence interpolated on $\N$ by $P$.
\end{proof}

\begin{Remark}\label{conversion}
Notice that a linear operator $T\in B(X)$ is an $(m,p)$-isometry if, and only if, for every $x\in X$, 
$(\norm{T^n x})_{n \in \N} \in \A_{m,p}$. Moreover, $T$ is a strict $(m,p)$-isometry if, in addition, 
for some $x_0\in X$, $(\norm{T^n x_0})_{n \in \N} \in \widehat{\A}_{m,p}$.
\end{Remark}

These facts allow us to retrieve most of the basic properties of $(m,p)$-isometries in an elementary way. For example,
Remark~\ref{hatrmk} shows that an $(m,p)$-isometry is an $(m+1,p)$ isometry.
Proposition~\ref{power bounded} follows from the fact that a polynomial which is bounded on $\N$ is constant.
We can give a simple unified proof for the following reproducing formulae for an $(m,p)$-isometry. Part (i) essentially
appears as \cite[Theorem~2.1]{bemarmart} while part (ii), as already mentioned, appears in \cite[Proposition~2.1]{ba}.

\begin{Proposition}
Let $T \in B(X)$ be an $(m,p)$-isometry.  Then for all $n\in \N^+$ with $n\geq m$ and for all $x \in X$: 
\begin{enumerate}[(i)]
  \item 
  	$\displaystyle \norm {T^n x}^p = \sum_{k=0}^{m-1} (-1)^{m-1-k}\binom {n}k\binom{n-1-k}{m-1-k} \norm{T^k x}^p$,
  \item 
  	$\displaystyle \norm {T^n x}^p = \sum_{k=0}^{m-1}n^{(k)}\beta_k^{(p)}(T,x)$,
  \item 
  	$\displaystyle \norm {T^n x}^p = \frac{\sum_{k=0}^{m-1}(-1)^{m-1-k}\binom{m-1}k \frac1{n-k}\norm{T^kx}^p}
     {\sum_{k=0}^{m-1}(-1)^{m-1-k}\binom{m-1}k \frac1{n-k}}$.
\end{enumerate}
\end{Proposition}
\begin{proof}
As $T$ is an $(m,p)$-isometry the sequence $a_n = \norm{T^n x}^p$ is interpolated by some polynomial $P\in \P^{m-1}$.
Evidently, $P$ must also be the (unique) Lagrange polynomial of degree less than or equal to $m-1$ which interpolates
$\{(k, a_k) \mid k=0,1,\ldots,m-1\}$. Using the normal form of the Lagrange polynomial to calculate $P(n)$ yields (i),
using the Newton form yields (ii), while using the barycentric form we get (iii).
\end{proof}

As one more example, we give an alternative derivation of
Proposition~\ref{inv} which states that if $T\in B(X)$ is an invertible $(m,p)$-isometry and
$m$ is even, then $T$ is an $(m-1,p)$-isometry.

\begin{proof}[Proof (of Proposition~\ref{inv})]
Fix $x\in X$ and choose $P\in \P^{m-1}$ with $P|_\N=(\norm{T^n x}^p)_{n\in \N}$.  Since $T$ is invertible, we actually
have $P|_{\mathbb{Z}}= (\norm{T^n x}^p)_{n\in \mathbb Z}$.  (Indeed, if $l$ is a negative integer then
the sequence $(\norm{T^{n}(T^lx)}^p)_{n\in \N}$ is interpolated on $\N$ by a polynomial $Q\in \P^{m-1}$ but this
sequence has $(\norm{T^n x}^p)_{n\in \N}$ as a subsequence and so, by uniqueness of polynomials, $Q$ must be a
translation of $P$, that is $Q(t)=P(t+l)$ for all $t\in \R$ and so $\norm{T^l x}^p = Q(0)= P(l)$ as claimed.)
In particular, $P$ is positive at every (negative) integer and thus cannot have odd degree.  Since $P\in \P^{m-1}$ and
$m-1$ is odd, we actually have $P\in \P^{m-2}$.  As $x$ was arbitrary in $X$, it follows that $T$ is an 
$(m-1,p)$-isometry.
\end{proof}

\section{$(m, p)$- and $(\mu, q)$-isometries}

We are interested in determining when an $(m,p)$-isometry is simultaneously a $(\mu,q)$-isometry, where 
$m, \mu \in \N^+$ and $p, q \in (0, \infty)$.  We have already seen that an $(m,p)$-isometry is an
$(m+k,p)$-isometry, for all $k \in \mathbb{N}$ and so it is natural to rephrase the problem in terms of strict 
$(m,p)$-isometries.  We begin with the following.

\begin{Proposition}\label{formula for q}
Suppose $m,\mu\in\N^+$ and $p,q \in (0,\infty)$ and $\widehat{\A}_{m,p}\cap\widehat{\A}_{\mu, q}\ne\emptyset$. 
Then $(m-1)q = (\mu-1)p$.
\end{Proposition}
\begin{proof}
Suppose $a\in\widehat{\A}_{m,p}\cap\widehat{\A}_{\mu, q}$. From Corollary \ref{polyTheorem} we know that there exist
polynomials $P$ and $Q$ with $P|_{\mathbb{N}}=a^p$ and $Q|_{\mathbb{N}}=a^q$, and (by~\ref{prop:strictdegree}) that
these polynomials have degree $m-1$ and $\mu-1$ respectively. Clearly $m=1$ forces $\mu=1$ and vice versa. So assume now
$m, \mu >1$. Then $\lim_{n\to\infty}a_n^p/n^{m-1}$ exists and is positive, being the leading
coefficient of $P$ and, similarly, $\lim_{n\to\infty}a_n^q/n^{\mu-1}$ exists and is positive (being the leading
coefficient of $Q$). Therefore the limits $\lim_{n\to\infty}a_n^{pq}/n^{q(m-1)}$ and $\lim_{n\to\infty}a_n^{pq}/n^{p(\mu-1)}$ both
exist and are positive, forcing the powers of $n$ in their denominators to be equal.
\end{proof}

\begin{Lemma}\label{gcdprop}
Suppose $a \in \widehat{\A}_{m,p}$ and $a \in \widehat{\A}_{\mu,q}$, with $m,\mu>1$. 
Then $a\in\widehat{\A}_{d+1,r}$, where $d=\gcd(m-1,\mu-1)$ and $r=\frac{d}{(m-1)} p$.
\end{Lemma}
\begin{proof}
Let $P$ and $Q$ be the polynomials of degree $m-1$ and $\mu-1$ respectively 
(\ref{polyTheorem}, \ref{prop:strictdegree}) that satisfy $P|_{\mathbb{N}}=a^p$ and $Q|_{\mathbb{N}}=a^q$. Then
$P^{\mu-1}|_{\mathbb{N}}=a^{p(\mu-1)}$ and $Q^{m-1}|_{\mathbb{N}}=a^{q(m-1)}$, so that, since $p(\mu-1)=q(m-1)$
by \ref{formula for q}, we have $P^{\mu-1}|_{\mathbb{N}}=Q^{m-1}|_{\mathbb{N}}$ and then $P^{\mu-1}=Q^{m-1}$, because a
polynomial is fully determined by its values on $\mathbb N$.  So $P$ and $Q$ have the same zeroes. If a zero
has multiplicity $\xi$ in $P$ and $\zeta$ in $Q$, then $\zeta(m-1)=\xi(\mu-1)$, so that the integer 
$(m-1)/d$ divides $\xi$. This is true for all zeroes of $P$, so $P=R^{(m-1)/d}$
where $R$ is a polynomial of degree $d$. Then $a^{p d/(m-1)}=R|_{\mathbb N}$ and \ref{polyTheorem} and
\ref{prop:strictdegree} give $a\in\widehat{\A}_{d+1,r}$.
\end{proof}

\begin{Proposition}\label{kp}
Suppose $a\in\A^+$. Then exactly one of the following occurs:
\begin{enumerate}[(i)]
	\item
		$\hat\rho(a)=\rho(a)=\emptyset$;
	\item
		$\hat\rho(a)=\{1\}\times(0,\infty)$ and $\rho(a)=\N^+\times(0,\infty)$;
	\item
		there exist unique $m_0\in\N\backslash\{0,1\}$ and $p_0\in(0,\infty)$ such that\\
		$\hat\rho(a)=\{(k(m_0-1)+1,kp_0)\mid k\in\N^+\}$ and\\
		$\rho(a)=\{(m,kp_0)\mid  k\in\N^+, m>k(m_0-1)\}$. 
\end{enumerate}
\end{Proposition}
\begin{proof}
It is evident that $\rho(a)$ is empty if and only if $\hat\rho(a)$ is empty and
that (ii) occurs if, and only if, $a$ is a constant sequence, so we suppose that $a$ is not constant and that
$\hat\rho(a)\ne\emptyset$ and show that (iii) occurs.\\
Let $m_0$ be the least integer such that $(m_0,p_0)\in\hat\rho(a)$ for some $p_0\in(0,\infty)$. Then, since $a$ is not
constant, $m_0>1$ and $p_0$ is unique by \ref{formula for q}.  Now $a \in \widehat{\A}_{m_0,p_0}$ and, by
\ref{prop:strictdegree}, the real polynomial $P$ with $P|_\N=a^{p_0}$ satisfies $\deg P = m_0-1$. Then for each
$k \in \mathbb{N}^+$, $a^{kp_0}=P^k|_\N$, with $\deg P^k = k(m_0-1)$ and we invoke \ref{polyTheorem} and
\ref{prop:strictdegree} again to get $\{(k(m_0-1)+1,kp_0)\mid k\in\N^+\}\subseteq \hat\rho(a)$. If conversely 
$(\mu,q) \in \hat\rho(a)$, by \ref{gcdprop} $(d+1,r) \in \hat\rho(a)$ for $d = \gcd(m_0-1,\mu-1)$. The minimality of
$m_0$ forces $d=m_0-1$ and then there exists a $k \in \N^+$ with $\mu-1=k(m_0-1)$. The form of $q$ follows from
\ref{formula for q}. So $\hat\rho(a)=\{(k(m_0-1)+1,kp_0)\mid k\in\N^+\}$, as required. Observing \ref{hatrmk}, we deduce
easily that $\rho(a)=\{(m,kp_0)\mid  k\in\N^+, m>k(m_0-1)\}$.
\end{proof}

Our aim is now to translate these results into the language of $(m,p)$-isometric operators by considering the sequences
$(\norm{T^nx}^p)_{n \in \N}$ for a given operator $T$ and $x\in X$. For $m \in \N^+$ and $p \in (0, \infty)$, it is
useful to define the subsets $X^T_{m,p}$ of $X$ by 
$X^T_{m,p}:=\{x\in X \mid (\norm{(T^nx)})_{n \in \N} \in \widehat\A_{m,p}\}$. Let's note some basic properties of these
sets before we state a fundamental decomposition theorem for $(m,p)$-isometric operators.

\begin{Lemma}\label{equpart}
Suppose  $m,\mu\in\N^+$ and $p,q\in(0,\infty)$.
Then, with reference to the operator $T \in B(X)$,
\begin{enumerate}[(i)]
	\item
		$X^T_{1,p}=X^T_{1,q}$;
	\item
		$X^T_{m,p}\subseteq X^T_{k(m-1)+1,kp}$ for all $k\in\N^+$;
	\item
		if $X^T_{m,p}\cap X^T_{\mu,q}$ is not empty, then $(m-1)q=(\mu-1)p$ and, provided $m,\mu>1$, 
		$X^T_{m,p}\cap X^T_{\mu,q}=X^T_{d+1,r}$, where $d=\gcd(m-1,\mu-1)$ and 
		$r={d\over m-1}p={d\over \mu-1}q$;
	\item
		if $m>1$ and $X^T_{m,p}\cap X^T_{m,q}\ne\emptyset$, then $p=q$.
\end{enumerate}
\end{Lemma}
\begin{proof}
\begin{enumerate}[(i)]
	\item
		$X^T_{1,p}=\{x\in X\backslash\{0\}\mid \|T^nx\|^p=\|x\|^p, \ \forall n\in\N\}$.
	\item 
		If $m=1$, this is clear, so suppose $m>1$, let $x \in  X^T_{m,p}$ and set $a := (\norm{T^nx})_{n \in \N}$.
		Then $(m,p)\in\hat\rho(a)$ and, using the
		representation of $\hat\rho(a)$ given in \ref{kp}, it follows that $(k(m-1)+1,kp)\in\hat\rho(a)$ also, so that
		$x \in X^T_{k(m-1)+1,kp}$. So $X^T_{m,p}\subseteq X^T_{k(m-1)+1,kp}$.
	\item 
		Suppose $X^T_{m,p}\cap X^T_{\mu,q}\ne\emptyset$; then $(m-1)q=(\mu-1)p$ by \ref{formula for q} and, if 
		$m,\mu>1$, then $X^T_{m,p}\cap X^T_{\mu,q}\subseteq X^T_{d+1,r}$ by \ref{gcdprop} and the reverse inclusion is
		got from (ii).
	\item This follows from the equation $(m-1)q=(m-1)p$ in (iii). 
\end{enumerate}
\end{proof}

\begin{Theorem}\label{union}
$T \in B(X)$ is a strict $(m,p)$-isometry if, and only if, there is an increasing finite sequence 
$(\nu_i)_{1\leq i\leq n}$ in $\N^+$ with $\nu_n=m$ such that the sets $X^T_{\nu_i,p}$ are non-empty and satisfy
\begin{align*}
	X^T_{\nu_1,p} \ \dot{\cup} \ ... \ \dot{\cup} \ X^T_{\nu_n,p} \ \dot{\cup} \ \{0\} = X.
\end{align*}  
This partition is independent of the parameters, that is to say, if $T$ is also a strict $(\mu,q)$-isometry and its
associated increasing sequence is $(\lambda_i)_{1\leq i\leq n'}$, then $n=n'$ and $X^T_{\nu_i,p}=X^T_{\lambda_i,q}$ for
$1\leq i\leq n$; moreover, $\lambda_i=1+{(\nu_i-1)q\over p}$ for $1\leq i\leq n$.
\end{Theorem}
\begin{proof}
Remarks~\ref{hatrmk} and \ref{conversion} argue the necessity of the
existence of such a partition and its sufficiency is clear from the
definitions.  Suppose now that $T$ is also $(\mu,q)$-strict.  Notice
that if  $x\in X^T_{\nu_r,p}$ then $x\in X^T_{\lambda_s,q}$ for some
$\lambda_s\in\{1,\ldots,\mu\}$ and Proposition~\ref{formula for q}
guarantees that $(\nu_r-1)q=(\lambda_s-1)p$.  In particular, we get
the same $\lambda_s$, namely $1+{(\nu_r-1)q\over p}$, for all $x\in
X^T_{\nu_r,p}$ and so $X^T_{\nu_r,p}\subseteq X^T_{\lambda_s,q}$.  By
symmetry, $X^T_{\lambda_s,q}\subseteq X^T_{\nu_t,p}$ for some $t$ and
the disjoint nature of partitions implies $t=r$ and therefore
$X^T_{\nu_r,p}=X^T_{\lambda_s,q}$.  So the two partitions are identical;
their specified orders also match because the increasing order of
$(\nu_i)_{1\leq i\leq n}$ is mirrored in the order of
$\left(1+{(\nu_i-1)q\over p}\right)_{1\leq i\leq n}$.  Therefore
$X^T_{\nu_i,p}=X^T_{\lambda_i,q}$ for $1\leq i\leq n$.
\end{proof}

We use \ref{equpart} and \ref{union} freely in the following corollaries.

\begin{Corollary}\label{multiup}
Suppose $T \in B(X)$ is  a strict $(m,p)$-isometry. Then $T$ is also a strict $(k(m-1)+1,kp)$-isometry for all
$k\in\N^+$.
\end{Corollary}
\begin{proof}
Suppose $k\in\N^+$ and let $(\nu_i)_{1\leq i\leq n}$ be the sequence
associated with $T$ as an $(m,p)$-isometry.  We
have $X^T_{\nu_i,p}\subseteq X^T_{k(\nu_i-1)+1,kp}$ for $1\leq i\leq n$ by \ref{equpart}. 
Since $\{X^T_{\nu_i,p}\mid 0\leq i\leq n\}$ is a (disjoint) partition of $X\backslash\{0\}$, 
the inclusions are in fact not proper and $\{X^T_{k(\nu_i-1)+1,kp}\mid 1\leq i\leq n\}$ is the same partition,
making $T$ a strict $(k(m-1)+1,kp)$-isometry.
\end{proof}

\begin{Corollary}\label{gcdT}
Suppose $T \in B(X)$ is a strict $(m,p)$-isometry and a strict $(\mu,q)$-isometry. Then $(m-1)q=(\mu-1)p$ and, 
provided $m,\mu>1$, $T$ is a strict $(d+1,r)$-isometry, where $d=\gcd(m-1,\mu-1)$ and $r=\frac{d}{m-1}p$.
\end{Corollary}
\begin{proof}
$m=1$ forces $\mu=1$ and vice versa. Assume now $m, \mu >1$.
Let $(\nu_i)_{1\leq i\leq n}$ and $(\lambda_i)_{1\leq i\leq n}$ be the sequences associated with $T$
as a strict $(m,p)$-isometry and a strict $(\mu,q)$-isometry,
respectively. Then $(\lambda_i-1)p=(\nu_i-1)q$ and in particular $(\mu-1)p=(m-1)q$. 
Further, by \ref{equpart}, $X^T_{1,p}=X^T_{1,q}=X^T_{1,r}$ and
$X^T_{\nu_i,p}=X^T_{\lambda_i,q}=X^T_{\delta_i+1,s_i}$,
where $\delta_i=\gcd(\nu_i-1,\lambda_i-1)$ and $s_i=\frac{\delta_i}{(\nu_i-1)}p$, for $i= 2,...,n$.  
Now, since $\lambda_n=\mu$ and $\nu_n=m$, the equations $(\lambda_i-1)p=(\nu_i-1)q$ give 
$(\lambda_i-1)(m-1)=(\nu_i-1)(\mu-1)$. Therefore, 
$\frac{(\lambda_i-1)}{\delta_i}\cdot \frac{(m-1)}{d}= \frac{(\nu_i-1)}{\delta_i}\cdot \frac{(\mu-1)}{d}$ 
and unique prime factorization (except if $i=1=\nu_i$) gives
$\frac{(m-1)}{d}=\frac{(\nu_i-1)}{\delta_i}$. Thus, $s_i=r$ and, moreover,
$(\delta_i+1)_{1\leq i\leq n}$ is increasing; in particular $d+1=\delta_n+1$ is the maximum value in
$(\delta_i+1)_{1\leq i\leq n}$. So 
$X^T_{1,r} \ \dot{\cup} \ ... \ \dot{\cup} \ X^T_{\delta_i+1,s_i} \ \dot{\cup} \ ... \ \dot{\cup} \ 
X^T_{d+1,r} = X \setminus \{0\}$ and $T$ is a strict $(d+1,r)$-isometry.
\end{proof}

Note the necessity of $(m-1)q=(\mu-1)p$ in \ref{gcdT} implies immediately that if $m>1$, a strict
$(m,p)$-isometry cannot be a strict $(m,q)$-isometry for $q \neq p$. Further, simultaneous $(m,p)$-strictness and
$(\mu,q)$-strictness with $\mu > m$ ($\mu < m$) requires $q > p$ ($q < p$). In particular, if
$m>1$, an $(m,p)$-isometry which is also an $(m,q)$-isometry for $q>p$ must be an $(m-1,p)$-isometry.

Notice further, that if $m-1$ and $\mu-1$ in \ref{gcdT} are
relatively prime then the partition of $X$ is 
$X=X^T_{m,p}\ \dot{\cup} \ X^T_{1,p} \ \dot{\cup} \ \{0\}=X^T_{\mu,q} \ \dot{\cup} \ X^T_{1,q} \
\dot{\cup} \ \{0\}$. 
Indeed, assuming $m$ and $\mu$ are at least 2,\ $X^T_{\nu,p}=X^T_{\lambda,q}$ implies
$\frac{\lambda-1}{\nu-1}=\frac{\mu-1}{m-1}$ ($=\frac qp$) and then 
$2 \leq \nu < m$, $2\leq \lambda < \mu$ would imply a common factor of
$\mu-1$ and $m-1$. We elaborate on this theme in the next result,
producing several other ways of determining whether or not we have
found the smallest $m$ such that $T$ is an $(m,p)$-isometry (for some $p$).

\begin{Corollary}\label{comfacseq}
Suppose $T \in B(X)$ is a strict $(m,p)$-isometry with associated sequence
$(\nu_i)_{1\leq i\leq n}$. Suppose that the terms of $(\nu_i-1)_{1\leq i\leq n}$ greater than $0$ 
have no common prime factor and that there are at least two such terms. Then for each $\mu \in \N^+$ such that $T$ is a
$(\mu, q)$-isometry, $\mu-1$ is an integer multiple of $m-1$ and, in particular, $m\leq\mu$.
\end{Corollary}
\begin{proof}
Suppose $T$ is a $(\mu,q)$-isometry. We may assume that it is strict
and let $(\lambda_i)_{1\leq i\leq n}$ be the associated
sequence. The hypothesis forces $m,\mu>1$, so $p/q=(m-1)/(\mu-1)$ is
rational; write $p/q$ as $u/v$ where $u,v\in\N^+$ and $\gcd(u,v)=1$.
For $\nu_i>1$, we have $(\lambda_i-1)p=(\nu_i-1)q$ and therefore
$(\lambda_i-1)u=(\nu_i-1)v$, ensuring that $u$ divides
$\nu_i-1$. Since $\nu_i>1$ is arbitrary, our hypothesis forces $u=1$
and then $\mu-1=v(m-1)$.
\end{proof}

There is nothing strange about $\mu-1$ being a multiple of $m-1$ in
\ref{comfacseq}. The result identifies $m$ as the least integer for
which $T$ is a $(m,p)$-isometry for any $p$. If, for example, the associated sequence has two consecutive
integers greater than $1$ as terms, in particular if
$X^T_{m-1,p}\ne\emptyset$, or if there are sufficiently many terms in
the sequence ($n>m/2+1$), or if $X^T_{2,p}\ne\emptyset$, 
then $m$ is that smallest number. The
reason why it is important to identify this minimum value is shown in the next result.

\begin{Theorem}\label{summarize}
Suppose $T \in B(X)$ is a strict $(m,p)$-isometry for some $m>1$ and
$p\in(0,\infty)$. Then there exist $m_0>1$ and $p_0\in(0,\infty)$
such that $T$ is a strict $(\mu,q)$-isometry if, and only if, $(\mu,q)=(k(m_0-1)+1,kp_0)$ 
for some $k \in \mathbb{N}^+$.
\end{Theorem}
\begin{proof}
Let $m_0$ be the least integer for which there exists $p_0$ such
that $T$ is an $(m_0,p_0)$-isometry.  Sufficiency is in
\ref{multiup}. Towards necessity, assume $T$ is $(\mu,q)$-strict.
Then, by \ref{gcdT}, $T$ is $(d+1,r)$-strict for $d = \gcd(m_0-1,\mu-1)$. 
Hence, $d+1 \leq m_0$, and the minimality of $m_0$ forces $d=m_0-1$. So, there exists
a \ $k \in \N^+$ with $\mu-1=k(m_0-1)$ and $q=kp_0$ then follows from \ref{gcdT}.
\end{proof}

Let us consider again Examples \ref{example finite} and \ref{example(2,p)}
in light of these results.

\begin{Example}
\begin{itemize}
	\item[(i)] 
		$T = \left(\begin{array}{cc}
						1 & 1\\
						0 & 1\\
              \end{array} \right)$ on $(\mathbb{C}^{2}, \pnorm{.}{2})$ is a strict $(3,2)$-isometry.
         Strictness can be asserted because otherwise $T$ would be a 
         $(2,2)$-isometry, and then by Proposition~\ref{inv} a $(1,2)$-isometry, 
         hence an isometry, which it is not. For the same reason, the index $3$ is minimal in the sense of 
         Theorem \ref{summarize}. By Theorem \ref{summarize}, $T$ is a strict $(2k+1,2k)$-isometry, 
         for any $k \in \mathbb{N}^+$, and is not a $(\mu,q)$-isometry for any other pair $(\mu,q)$.
	\item[(ii)] 
		$T_{p} \in B(\ell_{p})$ in Example \ref{example(2,p)} is a strict $(2,p)$-isometry, not a $(2,q)$-isometry,
		for any $q \neq p$, but a strict $(m,(m-1)p)$-isometry, for any $m \in \mathbb{N}$, $m \geq 2$ 
		(and for no other pair $(\mu,q)$).
\end{itemize}
\end{Example}

\section{$(m, \infty)$-isometric operators}

In this section, motivated by the consideration of $\ell_{\infty}$ in Example \ref{example(2,p)}, we give a natural extension of the definition of $(m,p)$-isometric operators where we now allow the second parameter $p$ to
become infinite.

Consider an $(m, p)$-isometry $T \in B(X)$. We have the following obvious equivalence for all $x \in X$:
\begin{align*}
	\sum_{k=0}^{m}(-1)^{k}{m \choose k}\| T^{k}x \| ^{p} &= 0 \\ 
	\Leftrightarrow \ \
	\bigg(\sum_{\substack{k \leq m \\ k \ \textrm{even}}}{m \choose k}\| T^{k}x \| ^{p}\ \bigg)^{1/p} 
	&= \bigg(\sum_{\substack{k \leq m \\ k \ \textrm{odd}}}{m \choose k}\| T^{k}x \| ^{p}\ \bigg)^{1/p}.
\end{align*}

Taking the limit as $p$ tends to infinity, prompts the following.

\begin{Definition}\label{def}
An operator $T \in B(X)$ is called an \emph{$(m, \infty)$-isometry}, if
\begin{align}
	\max_{\substack{k=0,...,m \\ k \ \textrm{even}}}\|T^{k}x\| 
	= \max_{\substack{k=0,...,m \\ k \ \textrm{odd}}}\|T^{k}x\|, \
    \ \ \forall x \in X. \label{minf}
\end{align}
\end{Definition}

Clearly a $(1,\infty)$-isometry is an isometry and vice versa. For
$m> 1$ however, condition~\eqref{minf} does not impose boundedness on a linear map $T$.
We will see that $(m,\infty)$-isometries share many of the basic
properties of $(m,p)$-isometries even though the two classes have an
essentially trivial intersection.

It is obvious on replacing $x$ in \eqref{minf} by $T^{\ell}x$ for
$\ell\in\N$ that $T$ is an $(m,\infty)$-isometry if, and only if, for
all $x\in X$,
\begin{align}\label{def (m,inf) sequences}
	\max\limits_{\substack{k=\ell,...,m+\ell \\ k \ \textrm{even}}} a_k 
	= \max\limits_{\substack{k=\ell,...,m+\ell \\ k \ \textrm{odd}}} a_k , \ \forall \ell \in \N,
\end{align}
where the sequence $a$ is given by $(a_k)_{k\in\N}:=(\|T^k x\|)_{k\in\N}$.\\

We begin our treatment by stating a major property of this kind of operator.

\begin{Theorem}\label{main theorem for (m,inf)}
Let $T \in B(X, \|.\|)$ be an $(m,\infty)$-isometry. Then there exists a norm on $X$, equivalent to $\|.\|$,
under which $T$ is an isometry.
\end{Theorem}

This condition is not sufficient for an operator to be an $(m, \infty)$-isometry (see Remark \ref{not sufficient}).
Before we can prove this theorem, we need some preliminary results. The following lemma provides an alternative  
description of $(m,\infty)$-isometries. Throughout, $\pi(k)=k \mathop{\mathrm{mod}} 2$ denotes the parity of $k\in\N$.

\begin{Lemma}\label{main lemma for (m,inf)}
Let $a\in\A$ and $m \in \N^+$. Then the following are equivalent.
\begin{itemize}
	\item[(i)]
		$a$ satisfies (\ref{def (m,inf) sequences}), i.e.    
		$\max\limits_{\substack{k=\ell,...,m+\ell \\ k \ \textrm{even}}} a_k 
		= \max\limits_{\substack{k=\ell,...,m+\ell \\ k \
        \textrm{odd}}} a_k$, \ $\forall \ell \in \N$\\[1ex]
	\item[(ii)]
 		$a$ attains a maximum and \ 
		$\max\limits_{k \in \N} a_k = \max\limits_{\substack{k= \ell,...,m-1+\ell \\ \pi(k)=\pi(m-1+\ell)}} a_k$, 
		\ $\forall \ell\in \N$.
\end{itemize}
\end{Lemma}
\begin{proof}
If $a$ satisfies (ii), then, for all $\ell\in \N$,
\[ \max\limits_{k \in \N} a_k 
= \max\limits_{\substack{k= \ell,...,m-1+\ell \\ \pi(k)=\pi(m-1+\ell)}} a_k 
\leq \max\limits_{\substack{k= \ell,...,m+\ell \\ \pi(k)=\pi(m-1+\ell)}} a_k 
\leq \max\limits_{k \in \N} a_k \]
and also
\[ \max\limits_{k \in \N} a_k 
= \max\limits_{\substack{k= \ell+1,...,m+\ell \\ \pi(k)=\pi(m+\ell)}} a_k
\leq \max\limits_{\substack{k= \ell,...,m+\ell \\ \pi(k)=\pi(m+\ell)}} a_k
\leq \max\limits_{k \in \N} a_k\]
and the ensuing equalities give 
$\max\limits_{\substack{k= \ell,...,m+\ell \\ \pi(k)=\pi(m-1+\ell)}} a_k
=\max\limits_{\substack{k= \ell,...,m+\ell \\ \pi(k)=\pi(m+\ell)}} a_k$; 
in other words, $a$ satisfies (\ref{def (m,inf) sequences}).

For the reverse implication, suppose $a$ satisfies (\ref{def (m,inf)
sequences}) and fix $n>m$.  By virtue of (\ref{def (m,inf) sequences}),  
$\max\limits_{k=n-m,...,n}a_k$ is attained for at least two indices
(one even and one odd). Hence, there exists an $r < n$, such that
$a_r \geq a_n$. It follows that the sequence indeed has a maximum and
that it is attained over the first $m$ indices, that is,
$\max\limits_{k \in \N}a_k=\max\limits_{k = 0,...,m-1}a_k =\max\limits_{k = 0,...,m}a_k$. 
This last maximum is achieved at either an even or odd index, but in accordance with 
(\ref{def (m,inf) sequences}), it is achieved both over even
and over odd indices.  For later reference, remark that this means the
maximum of the sequence $a$ is unchanged if the $a_0$ term is discarded. In any case, 
we can write,
\[ \max\limits_{k \in \N}a_k=\max\limits_{k = 0,...,m}a_k
=\max\limits_{\substack{k = 0,...,m \\ \pi(k)=\pi(m-1)}} a_k
=\max\limits_{\substack{k = 0,...,m-1 \\ \pi(k)=\pi(m-1)}} a_k.\]
This establishes that the equation in (ii) holds for $\ell=0$.  

Now remark that property (\ref{def (m,inf) sequences}) is inherited by the subsequence
$(a_{k+1})_{k \in \N}=:a'$ obtained on discarding the $a_0$ term (or indeed any
finite number of the leading terms) from $a$ so that we can apply the
above argument to $a'$ (as already remarked above, the maximum of $a'$ is the same as that of $a$) 
to gain that the equation in (ii) holds for $\ell=1$.  This process can be repeated ad inifinitum
in order to assert (ii).
\end{proof}

Let us translate this into the language of $(m,\infty)$-isometric operators:

\begin{Corollary}\label{ell operators}
Let $T \in B(X)$. Then $T$ is an $(m,\infty)$-isometry if, and only if, for all $\ell \in \mathbb{N}$ and 
all $x\in X$,
\begin{align*}
	\max_{k\in\N}\|T^kx\| =\max\limits_{\substack{k= \ell,...,m-1+\ell \\ \pi(k)=\pi(m-1+\ell)}}\|T^{k}x\|
\end{align*}	
\end{Corollary}

\begin{Remark}\label{not sufficient}
The backward implication in
Lemma \ref{main lemma for (m,inf)} requires the parity statement
and may be false without it. Consider, for example, the sequence 
$(b_k)_{k\in\N}$ with $b_k=\pi(k)$. Clearly, for $m=2$,\  
$\max\limits_{k \in \N}b_k = \max\limits_{k =\ell,\ldots,\ell+m-1}b_k$ for all $\ell \in \N$, but $(b_k)_{k\in\N}$
fails (\ref{def (m,inf) sequences}) for any $m \in \N^+$.

As one might then expect, the parity statement is required for the
backward implication in \ref{ell operators} as well.
For example, consider $m=2$ and any non-isometric operator 
$S \in B(X)$ with $S^2=I$, say
$S = \begin{pmatrix}
	 4 & 5\\
	 -3 & -4
     \end{pmatrix} $ on $(\mathbb{C}^2,\pnorm{.}{\infty})$.
Such $S$ satisfies $\max(\|S^k x\|)_{k\in\N} = \max\{\|x\|,\|Sx\|\} = \max\{\|S^{\ell}x\|,\|S^{\ell +1}x\|\}$, 
for all $\ell \in \mathbb{N}$ and all $x \in X$. However, it follows directly from the definition that an 
$(m, \infty)$-isometry $T$ which satisfies $T^2=I$ is an isometry. Hence, $S$ is not a
$(2,\infty)$-isometry, nor indeed an $(m,\infty)$-isometry for any $m \in \mathbb{N}^+$.
\end{Remark}

An operator is called \emph{power bounded} if there exists $C > 0$ such that $\|T^n\|\leq C$ for all 
$n \in \mathbb{N}$. The following statement is a trivial consequence of \ref{ell operators}. 

\begin{Corollary}\label{n leq}
Let $T \in B(X)$ be an $(m,\infty)$-isometry. Then, for all $n \in \mathbb{N}$,
\begin{align*}
	\|T^{n}x\| \leq \max_{k = 0,..., m-1}\|T^{k}x\|, \ \ \forall x \in X.
\end{align*}
In particular, $T$ is power bounded by $C := \max_{k =0,..., m-1}\|T^{k}\|$.
\end{Corollary}

We can now easily prove the main result of this section.

\begin{proof}[Proof of Theorem \ref{main theorem for (m,inf)}]
Define $|.|:X \rightarrow [0, \infty)$, by $|x|= \max_{k =0,..., m-1}\|T^{k}x\|$. 
Since the maximum preserves the triangle inequality and $T$ is linear, $|.|$ is a norm on $X$.
Now, by definition $\|x\| \leq |x|$, for all $x \in X$. Furthermore, by Corollary \ref{n leq}, for  
$C := \max_{k =0,..., m-1}\|T^{k}\|$, we have
\begin{align*} 
	\|x\| \leq |x| \leq C \cdot \|x\|, \ \ \forall \ x \in X.	
\end{align*}
Thus, the two norms are equivalent. (In particular, $(X, |.|)$ is a
Banach space.)
Corollary \ref{ell operators} implies that $T$ is an isometry with respect to $|.|$. 
\end{proof}

\begin{Remark}
The use here of linearity and boundedness of $T$ is essential as the result obviously fails for 
non-continuous self-maps of $X$ satisfying \ref{minf}.  
\end{Remark}

Theorem \ref{main theorem for (m,inf)} and Corollary \ref{n leq}
enable us to  deduce easily basic properties of  $(m,\infty)$-isometric operators which coincide with 
those of $(m,p)$-isometries. We will do so in the next section. First we want to look at some examples of, 
and sufficiency conditions for, $(m,\infty)$-isometric operators.

\begin{Proposition}\label{construct1}
Let $T \in B(X)$ and $m\in\N$, $m\ge 2$.  Suppose that
\begin{align*} 
	\|T^{m}x\|=\|T^{m-1}x\| \ \ \textrm{and} \ \ \|T^{m}x\| \geq \|T^{k}x\|, \ \ k=0, ..., m-2, \ \forall x \in X.
\end{align*}
Then $T$ is an $(m, \infty)$-isometry.  If $m=2$ then this condition is also necessary.
\end{Proposition}
\begin{proof}
The first part follows from the definition \eqref{minf} of an $(m, \infty)$-isometry.\\
Let $T$ be an $(2, \infty)$-isometry. By definition $\|Tx\|=\max \{\|T^{2}x\|, \|x\|\}$, for all 
$x \in X$. Hence, $\|Tx\| \geq \|T^{2}x\|$ and $\|Tx\| \geq \|x\|$, for all $x \in X$. 
Replacing $x$ with $Tx$ yields $\|T^{2}x\|=\max \{\|T^{3}x\|, \|Tx\|\}$ and
therefore $\|T^{2}x\| \geq  \|Tx\|$ for all $x \in X$. So we have equality, which proves the statement.
\end{proof}

For $m > 2$ the above condition is not a necessary one for an
$(m,\infty)$-isometry, as is demonstrated by Example \ref{m-inf finite}.

\begin{Proposition}\label{construct2}
Let $T \in B(X)$ and $n \in \mathbb{N}$ be odd. If $T^n$ is an isometry then $T$ is a
$(m, \infty)$-isometry for all $m \geq 2n-1$.
\end{Proposition}
\begin{proof}
Suppose $T^n$ is an isometry. Then $\|T^kx\|=\|T^{n+k}x\|$ for all $x \in X$, for all $k \in \mathbb{N}$. 
If $k$ is even, $n+k$ is odd and vice versa. Hence,
\begin{align*}
	\{\|T^{k}x\| : k\in \{0,...,2n-1\} \ \textrm{even}\} = \{\|T^{k}x\| : k\in \{0,...,2n-1\} \ \textrm{odd}\}
\end{align*}
and $T$ is a $(2n-1, \infty)$-isometry.  Similarly  (or by~\ref{m+1}), one sees
that $T$ is an $(m, \infty)$-isometry, for all $m \geq 2n-1$. 
\end{proof}

\begin{Remark}\label{S}
Proposition \ref{construct2} is not in general true for even $n$. Consider again a non-isometric 
operator $S$, with $S^2=I$. In particular, $S^2$ is an isometry, but $S$ cannot be a $(3,\infty)$-isometry as this
would imply it is an isometry (Remark~\ref{not sufficient}).
\end{Remark}

\begin{Example}
Let $m \in \mathbb{N}$, $m \geq 1$, $p \in [1, \infty]$ and $T \in B(\ell_{p})$ be a weighted right-shift
operator with a weight sequence $(\lambda_{n})_{n \in \mathbb{N}}\subset \mathbb{C}$ such that
\begin{align*}
	|\lambda_{n}| \geq 1, \ \ \textrm{for} \ \ n=1,...,m-1, \ \  
	\textrm{and} \ \ |\lambda_{n}|=1, \ \ \textrm{for} \ \ n \geq m.
\end{align*}
For readability, we will write $\mu_{n}=\lambda_{n}$, if $n \geq m$ (i.e., for those $\lambda_{n}$ which
are definitely of modulus $1$).
	
That is, for all $x=(x_{0}, x_{1}, x_{2}, x_{3}, ...) \in \ell_{p}$,
\begin{align*}
	(Tx)_{n} = \left\{ \begin{array} {lll}
				0, \ \ \ &n = 0, \\[1ex]
       			\lambda_{n}x_{n-1}, \ \ \ &n < m, \\[1ex]
       			\mu_{n}x_{n-1}, \ \ \ &n \geq m.
    			\end{array} \right.
\end{align*}
Hence,
\begin{align*}
	Tx = (&0, \lambda_{1}x_{0}, \lambda_{2}x_{1}, \lambda_{3}x_{2}, ...,\lambda_{m-1}x_{m-2}, \mu_{m}x_{m-1},
	\mu_{m+1}x_{m}, ...),\\
	T^{2}x = (&0, 0, \lambda_{2}\lambda_{1}x_{0}, \lambda_{3}\lambda_{2}x_{1}, \lambda_{4}\lambda_{3}x_{2}, ...,\\
	&\lambda_{m-1}\lambda_{m-2}x_{m-3}, \mu_{m}\lambda_{m-1}x_{m-2}, \mu_{m+1}\mu_{m}x_{m-1}, ...),\\
	\vdots \\
	T^{m-1}x = (&0, ..., 0, \lambda_{m-1} \cdots \lambda_{1}x_{0}, 
	\mu_{m}\lambda_{m-1} \cdots \lambda_{2}x_{1}, ...,\\
	&\mu_{2m-3}\cdots \mu_{m}\lambda_{m-1}x_{m-2}, \mu_{2m-2}\cdots \mu_{m}x_{m-1}, ...),\\
	T^{m}x = (&0, ..., 0, \mu_{m}\lambda_{m-1} \cdots \lambda_{1}x_{0},
	\mu_{m+1}\mu_{m}\lambda_{m-1} \cdots \lambda_{2}x_{1}, ...,\\
	&\mu_{2m-2}\cdots \mu_{m}\lambda_{m-1}x_{m-2},\mu_{2m-1}\cdots \mu_{m}x_{m-1},...),
\end{align*}
where the first $k$ coefficients of $T^kx$ are zero.\\
Since $|\lambda_{n}| \geq 1$, for $n < m$, and $|\mu_{n}|=1$, for $n \geq m$, we see that $T$ satisfies
$\|T^{m}x\| = \|T^{m-1}x\|$ and $\|T^{m}x\| \geq \|T^{k}x\|$, for $k=0,...,m-2, \ \forall x \in X$.
It is therefore an $(m, \infty)$-isometry by Proposition \ref{construct1}. Furthermore, it is obvious
that $T$ is not an isometry for $m\geq 2$, if there exists an $n_{0} \geq 1$, such that 
$|\lambda_{n_{0}}|\neq 1$.
\end{Example}

\begin{Example}\label{m-inf finite}
Let \ $T = \left(\begin{array}{cc}
			0 & 1\\
			-1 & 1\\
			\end{array} \right)$ 
be defined on $(\mathbb{K}^2, \|.\|)$, where $\|.\|$ denotes an arbitrary norm.
Since $T^3x=-x$ for all $x \in \mathbb{K}^2$, Proposition \ref{construct2} implies that $T$
is a $(5, \infty)$-isometry. It is easy to see that in general $T$ is not a
$(4, \infty)$-isometry (e.g. consider $x=\binom01$ and $\norm\cdot=\norm\cdot_1$). 
Thus, borrowing terminology from the previous section, we see $T$ is a \emph{strict} $(5,\infty)$-isometry.
\end{Example}

\section{$(m,\infty)$- and $(m,p)$-isometries}

Since $(m,\infty)$-isometries arise as an analogue of
$(m,p)$-isometric operators, two natural questions present themselves.
First, which properties of $(m,p)$-isometries are also enjoyed by
$(m,\infty)$-operators, and second, what is the setwise intersection
of these two classes.  The answer to the latter question is immediate
on comparing Corollary \ref{n leq} with Proposition \ref{power bounded}.

\begin{Proposition}
An $(m, p)$-isometry $T \in B(X)$ is not a $(\mu, \infty)$-isometry for any $\mu \geq 1$, 
$\mu \in \mathbb{N}$, unless it is an isometry.
\end{Proposition}

Nevertheless, we shall see that virtually all basic properties of $(m,p)$-isometric operators (as stated in \cite{ba}
and \cite{sa}) do hold for $(m,\infty)$-isometries.

\begin{Proposition}
\begin{itemize}
	\item[(i)] 
		Let $T \in B(X)$ be a $(1,\infty)$-isometry. Then $T$ is an isometry.
	\item[(ii)] 
		Let $T \in B(X)$ be an isometry. Then $T$ is an $(m,\infty)$-isometry for all 
		$m \in \mathbb{N}$, $m \geq 1$.
	\item[(iii)] 
		An $(m,\infty)$-isometry is bounded below (hence injective) and has therefore closed range.
	\item[(iv)] 
		If $T \in B(X)$ is a $(2,\infty)$-isometry, then $\|Tx\| \geq \|x\|$, for all $x \in X$.
\end{itemize}
\end{Proposition}
\begin{proof}
The only statement which might not be immediately clear from the
definitions is that an $(m,\infty)$-isometry is bounded below, but this follows directly from 
Theorem~\ref{main theorem for (m,inf)}. \footnote{One can also easily derive this directly from the definition,
by assuming the opposite and chosing a suitable sequence $(x_n)_{n \in \N}$ with $\|x_n\|=1$ and 
$\|T x_n\| \rightarrow 0$.}  
\end{proof}
 
\begin{Proposition}\label{m+1}
Suppose $T \in B(X)$ is an $(m,\infty)$-isometry. Then $T$ is an $(m+1,\infty)$-isometry.
\end{Proposition}
\begin{proof} 
By \ref{ell operators}, $\max\limits_{k\in\N}\|T^kx\|$ exists
and, for all $\ell\in\N$ and $x\in X$, we have
\begin{align*}
	\max_{k\in\N}\|T^kx\| 
	=\max\limits_{\substack{k= \ell+1,...,m+\ell \\ \pi(k) = \pi(m+\ell)}} \|T^{k}x\| 
	\leq\max\limits_{\substack{k= \ell,...,m+\ell \\ \pi(k)=\pi(m+\ell)}}\|T^{k}x\|\leq\max_{k\in\N}\|T^kx\|
\end{align*}	
and the ensuing equality gives the result by invoking \ref{ell operators} again.
\end{proof}

Let now $T \in B(X)$ be an invertible operator. For $\ell \in \mathbb{Z}$, replacing $x$ by
$T^\ell x$ in \eqref{minf} gives us that $T$ is an $(m,\infty)$-isometry if, and only if, 
for all (equivalently, for some) $\ell\in\mathbb Z$,
\begin{align}\label{def (m,inf) invertible operators}
	\max_{\substack{k=\ell,...,\ell+m \\ k \ \textrm{even}}} \|T^{k}x\| 
	= \max_{\substack{k=\ell,...,\ell+m \\ k \ \textrm{odd}}} \|T^{k}x\| \ \ \forall x \in X.
\end{align}
Following a mild modification of the proof, Lemma~\ref{main lemma for (m,inf)} remains valid for a 
$\Z$-indexed sequence $a$ if $\ell\in\N$ and $k\in\N$ are replaced respectively by $\ell \in \Z$ and $k\in\Z$.

\begin{Proposition}\label{m-1}
Suppose $T \in B(X)$ is an invertible $(m,\infty)$-isometry.
\begin{itemize}
	\item[(i)] 
		$T^{-1}$ is an $(m,\infty)$-isometry.
	\item[(ii)] 
		If $m$ is even, $T$ is an $(m-1, \infty)$-isometry.
\end{itemize}
\end{Proposition}
\begin{proof}
(i) This follows from (\ref{def (m,inf) invertible operators}) on setting $\ell = -m$.\\
(ii)  Suppose $m$ is even. By \ref{ell operators}, $\max_{k\in\N}\|T^kx\|$ exists
 and, for all $\ell\in\N$ and $x\in X$,  
\begin{align*}
	\max_{k\in\N}\|T^kx\| 
	=\max\limits_{\substack{k= \ell-1,...,m-2+\ell \\ \pi(k)=\pi(m-2+\ell)}} \|T^{k}x\|
	=\max\limits_{\substack{k= \ell,...,m-2+\ell \\ \pi(k)=\pi(m-2+\ell)}}\|T^{k}x\|,
\end{align*}
$\ell-1$ being valid in the middle expression because $T$ is
invertible and $\ell-1$ being dropped in the last because $\pi(\ell-1)\ne\pi(m-2+\ell)$
for even $m$. The result then follows by invoking \ref{ell operators} again.
\end{proof}

Example \ref{m-inf finite} shows that Proposition \ref{m-1}.(ii) does not hold in general for odd $m$.
Theorem \ref{main theorem for (m,inf)} allows us to easily deduce further similarities between $(m,p)$- and 
$(m,\infty)$-isometric operators. 

\begin{Proposition}\label{approximate point spectrum}
Let $X$ be a complex Banach space and $T \in B(X)$ an $(m,\infty)$-isometry. Then the approximate point
spectrum $\sigma_{\textrm{ap}}(T)$ of $T$ lies in the unit circle $\mathbb{T}$ and the spectral radius $r(T)$ 
is equal to $1$.
\end{Proposition}
\begin{proof}
The approximate point spectrum contains the boundary of the spectrum (see for example
\cite[Chapter VII. Proposition 6.7]{co}) and spectral properties of an operator do not change under
equivalent norms. Thus, the statement follows from Theorem \ref{main theorem for (m,inf)}. 
\end{proof}

\begin{Remark}\label{real spectrum}
The restriction to the complex case in \ref{approximate point spectrum} has actually no significance. 
One can prove that the approximate spectrum contains the boundary of the spectrum generally 
for bounded operators on real or complex Banach spaces.
\end{Remark}

Recall now that for a bounded operator $T \in B(X)$ and a subset $E \subseteq X$, the \emph{orbit} of $E$ under $T$, is
defined by Orb$(E, T):=\bigcup_{n \in \mathbb{N}}T^{n}(E)$. $T$ is called \emph{hypercyclic} if there exists an 
$x \in X$ such that Orb$(\{x\}, T)$ is dense in $X$ and $N$-\emph{supercyclic} if there exists a subspace 
$E \subseteq X$ with $\dim E$ $=N$, such that Orb$(E, T)$ is dense in $X$.\\
The following follows immediately from Theorem \ref{main theorem for (m,inf)}.

\begin{Proposition}
Let $T \in B(X)$ be an $(m, \infty)$-isometry. Then $T$ is not hypercyclic.
\end{Proposition}

However, analogous to the $(m, p)$-isometry case, we can do better. 
Bayart proves in \cite[Theorem 3.4]{ba} that an isometry on a complex infinite-dimensional Banach space is not
$N$-supercyclic, for any $N \geq 1$. Hence, Theorem \ref{main theorem for (m,inf)} implies: 

\begin{Proposition}\label{N-super}
On a complex infinite-dimensional Banach space $X$, an $(m, \infty)$-isometry $T \in B(X)$ 
is not $N$-supercyclic, for any $N \geq 1$.
\end{Proposition}

We conclude that almost all basic properties (Proposition \ref{power bounded} being a notable exception)
of $(m,p)$-isometric operators for $p \in (0, \infty)$ on Banach spaces are also valid for $p = \infty$.\\

\textbf{Acknowledgement}\\[1ex]
The authors wish to thank the referee for  suggestions which led to a
significant improvement of this paper.

$\textrm{}$\\

School of Mathematical Sciences, University College Dublin, Ireland\\

Philipp Hoffmann\\
\emph{email: philipp.hoffmann@ucdconnect.ie}\\

Michael Mackey\\
\emph{email: michael.mackey@ucd.ie}\\

Mícheál Ó Searcóid\\
\emph{micheal.osearcoid@ucd.ie}

\end{document}